\documentclass{article}

\usepackage[utf8]{inputenc}
\usepackage[T1]{fontenc}
\usepackage[english]{babel}
\usepackage{amsmath}
\usepackage{amsthm}
\usepackage{amssymb}
\usepackage{amsfonts}
\usepackage{graphicx}
\usepackage{placeins}
\usepackage[lofdepth,lotdepth]{subfig}
\usepackage{mathtools,mathrsfs} 
\usepackage{cases} 
\usepackage{tikz}

\newtheorem{theorem}{Theorem}[section]
\newtheorem{proposition}[theorem]{Proposition}
\newtheorem{definition}[theorem]{Definition}
\newtheorem{lemma}[theorem]{Lemma}

\theoremstyle{remark}
\newtheorem{remark}[theorem]{Remark}

\numberwithin{equation}{section}

\newcommand{\calC}{ C}

\newcommand{\calM}{\mathscr{M}}

\newcommand{\R}{\mathbb{R}}

\newcommand{\PP}{\mathbb{P}}

\newcommand{\calTT}{\mathcal{T}}

\newcommand\abs[1]{\left|#1\right|}

\newcommand{\tend}[2]{\displaystyle\mathop{\longrightarrow}_{#1\rightarrow#2}}

\newcommand\dom{d_{\partial \Omega}}

\DeclareMathOperator{\diam}{diam}
\def\widthfig{3.8}
\def\widthvoronoi{4.}

\title{Numerical computation of the cut locus via a variational approximation of the distance function}
\author{François Générau, Edouard Oudet, Bozhidar Velichkov}

\begin{document}
\maketitle
\begin{abstract}
  We propose a new method for the numerical computation of the cut locus of a compact submanifold of $\R^3$ without boundary. This method is based on a convex variational problem with conic constraints, with proven convergence. We illustrate the versatility of our approach by the approximation of Voronoi cells on embedded surfaces of $\R^3$.
\end{abstract}
\section{Introduction}

Let $S$ be a real analytic surface without boundary embedded in $\R^3$, and $b\in S$ any point of $S$ (that can be thought of as a base point).
\begin{definition}
  The \emph{cut locus} of $b$ in $S$ can be defined as the closure of the set of points $p\in S$ such that there exists at least two minimizing geodesics between $p$ and $b$. We will denote it by $Cut_b(S)$. Equivalently, it is also the set of points around which the distance function to the point $b$ - denoted by $d_b$ - is not smooth.
\end{definition}
\noindent The cut locus is a fundamental object in Riemannian geometry, and it is a natural problem to try and find ways to compute it numerically. In this paper, we propose a numerical approximation of $Cut_b(S)$, based on a convex variational problem on $S$, with proven convergence. It is not trivial to compute $Cut_b(S)$ because it is not stable with respect to $C^1$-small variations of $S$.
See for instance \cite[Example 2]{albano2016stability}. For instance, one can't approximate the cut locus of $S$ with the cut locus of a piecewise linear approximation of $S$.

\textbf{Related works.} Let us review the techniques used in the past by different authors to approximate the cut locus. We may divide them in two categories.

\emph{Geodesic approximation on parametrized surfaces.} This approach was used in \cite{sinclair2002loki} and \cite{misztal2011cut}. In \cite{sinclair2002loki}, on genus $1$ parametrized surfaces, the authors computed a degree $4$ polynomial approximation of the exponential map using the geodesic equation, and deduced an approximation of the cut locus from there. In \cite{misztal2011cut}, the authors used the deformable simplicial complexes (DSC) method and finite differences techniques for geodesic computations, to compute geodesic circles of increasing radius and their self intersection, \textit{i.e.} the cut locus.
They apply the method to genus $1$ surfaces. These papers contain no proof of convergence of the computed cut locus.

\emph{Exact geodesic computation on discretized surfaces.} This approach was used in \cite{itoh2004thaw} and \cite{dey2009cut}. In \cite{itoh2004thaw}, the authors computed the geodesics on a convex triangulated surface. They deduced an approximation of the cut locus of the triangulated surface, and filtered it according to the angle formed by the geodesics meeting at a point of the approximated cut locus, to make their approximation stable. They applied the method to ellipsoids. There is no proof of convergence. In \cite{dey2009cut}, the authors computed shortest curves on a graph obtained from a sufficiently dense sample of points of a surface. From there they deduced an approximation of the cut locus, and filtered it according to the maximal distance (called \emph{spread}) between the geodesics meeting at a point of the approximated cut locus. They proved that the set they compute converges to the cut locus (see \cite[Theorem 4.1]{dey2009cut}).

We may also mention \cite{bonnard2014geometric_cutlocus}, where the authors use some more geometric tools to compute (numerically) the cut locus of an ellipsoid, or a sphere with some particular metric with singularities.

\textbf{Our method}. The strategy we use is quite different. Given $m>0$ a constant, let $u_m$ be the minimizer of the following variational problem
\begin{equation}\label{eq:gradient constraint}
  \min_{\substack{u\in H^1(S)\\ \abs{\nabla_{_S} u} \leq 1 \\u(b)=0}} \int_{S} \left( \abs{\nabla_{_S} u}^2 - mu\right),
\end{equation}
where $\nabla_{_S}$ denotes the gradient operator on the surface $S$.
For $\lambda>0$ to be chosen small, we will use the set $E_{m,\lambda} := \left\{ \abs{\nabla_{_S} u_m}^2 \leq 1 - \frac{\lambda^2}{u_m^2}\right\}$ as an approximation of $Cut_b(S)$.
This is justified by some theoretical results regarding problem \eqref{eq:gradient constraint} obtained in \cite{generauCutLocusCompact2020}, which will be summarized in section \ref{sec:variational problem}. Now the set $E_{m,\lambda}$ can be well approximated using finite elements on a triangulation of the surface $S$.

The rest of the paper is organized as follows. In section \ref{sec:lambda medial axis}, we recall the notion of $\lambda$-medial axis that was introduced in \cite{chazal_lieutier_lambda}, and summarize some of its properties.
In section \ref{sec:lambda cut locus}, following the strategy of the $\lambda$-medial axis, we define a "$\lambda$-cut locus" $Cut_b(S)_{\lambda}$ and show that it can be used as an approximation of the complete cut locus.
In section \ref{sec:variational problem}, we recall the result from \cite{generauCutLocusCompact2020} which states that the set $E_{m,\lambda}$ defined above is a good approximation of $Cut_b(S)_{\lambda}$ if $m$ is big enough.
In section \ref{sec:discretization}, we discretize problem \eqref{eq:gradient constraint} using finite elements, to find a discrete minimizer $v_h$, where $h>0$ is the step of the dicretization. We show that the set
\begin{equation*}
  E_{m,\lambda,h} := \left\{x\in S \setminus\{b\} : \abs{\nabla_{_S} u_{m,h}^l(x)}^2 \leq 1 - \frac{\lambda^2}{(u_{m,h}^l)^2(x)}\right\},
\end{equation*}
is a good approximation of $E_{m,\lambda}$ as $h\to 0$.
In section \ref{sec:numerical}, we present the results of some numerical experiments.


\section{$\lambda$-Medial axis}\label{sec:lambda medial axis}
In this section, we recall briefly the notion of $\lambda$-medial axis introduced by Chazal and Lieutier in \cite{chazal_lieutier_lambda}.
Given an open subset $\Omega$ of $\R^2$, its medial axis $\mathcal{M}(\Omega)$ is defined as the set of points of $\Omega$ that have at least two closest points on the boundary $\partial \Omega$ of $\Omega$:
\[ \mathcal{M}(\Omega) := \big\{ x \in \Omega : \exists y,z \in \partial \Omega, \; y \neq z \; \text{and} \; \dom(x) = \abs{x-y} = \abs{x-z} \big\}, \]
where for any $x\in\Omega$, $\dom(x)$ is the distance from $x$ to the boundary $\partial\Omega$,
\[d_{\partial\Omega}(x)=\min\big\{|x-y|\ :\ y\in\partial\Omega\big\}.\]
The medial axis $\mathcal{M}(\Omega)$ is unstable with respect to small non-smooth perturbations of the boundary of $\Omega$. To deal with this issue, Chazal and Lieutier defined the so called {\it$\lambda$-medial axis} of $\Omega$ by setting, for any $\lambda>0$,
\begin{equation}
  \mathcal{M}_\lambda (\Omega):= \{x\in\Omega : r(x)\geq\lambda\},
\end{equation}
where $r(x)$ is the radius of the smallest ball containing the set of all closest points to $x$ on $\partial \Omega$, \textit{i.e.} the set $\{z\in \partial \Omega : \abs{x-z} = \dom(x) \}$. The map $\lambda \mapsto \mathcal{M}_\lambda (\Omega)$ is non increasing, and
\begin{equation*}
  \mathcal{M}(\Omega) = \bigcup_{\lambda>0}\mathcal{M}_\lambda (\Omega).
\end{equation*}
It is further proved in \cite[section 3, theorem 2]{chazal_lieutier_lambda} that $\calM_\lambda (\Omega)$ has the same homotopy type as $\calM (\Omega)$, for $\lambda$ small enough. These facts justify that $\calM_\lambda (\Omega)$ is a good approximation of $\calM(\Omega)$, for $\lambda$ small enough.
The crucial  difference though is that $\calM_\lambda (\Omega)$ is stable with respect to small variations of $\Omega$, whereas $\calM(\Omega)$ is not. We refer the reader to \cite[section 4]{chazal_lieutier_lambda} for precise statements and proofs.


\section{$\lambda$-Cut locus}\label{sec:lambda cut locus}
We want to define a set similar to the $\lambda$-medial axis, in the case of the cut locus $Cut_b(S)$. To this end, we note that, as it can be seen from \cite[section 2.1]{chazal_lieutier_lambda}, we have
\begin{equation}\label{eq:equivalentdef}
  \mathcal{M}_\lambda(\Omega) = \left\{x\in \Omega : \abs{\nabla \dom(x)}^2 \leq 1 - \frac{\lambda^2}{\dom^2(x)}\right\},
\end{equation}
where $\nabla \dom$ denotes the generalized gradient wherever $\dom$ is not differentiable. Analogously, for $\lambda>0$, we define the $\lambda$-cut locus as
\begin{equation*}
  Cut_b(S)_\lambda := \left\{x\in S \setminus \{b\} : \abs{\nabla_{_S} d_b(x)}^2 \leq 1 - \frac{\lambda^2}{d_b^2(x)}\right\}.
\end{equation*}
Note that, according to \cite[Proposition 3.4]{mantegazza_mennucci_2003}, the function $d_b$ is locally semiconcave on $S\setminus \{b\}$, so it has a generalized gradient everywhere on $S\setminus \{b\}$, whose norm is given by the following formula:
\begin{equation}
  \abs{\nabla_{_S} d_b}(x)=\max(0, \sup\limits_{v\in T_xS, \abs{v}=1} \partial^+_v d_b(x)).
\end{equation}
We have the following proposition from \cite[Proposition 2.9]{generauCutLocusCompact2020}.
\begin{proposition}\label{prop:increasing union}
  The map $\lambda \mapsto Cut_b(S)_\lambda$ is non increasing, and
  \begin{equation*}
    Cut_b(S) = \overline{\bigcup_{\lambda>0}Cut_b(S)_\lambda}.
  \end{equation*}
\end{proposition}

\noindent In addition, the following proposition holds.
\begin{proposition}\label{prop:homotopy type}
  If $S$ is a real analytic surface, then for $\lambda>0$ small enough, one of the connected component of $Cut_b(S)_\lambda$ has the same homotopy type as $Cut_b(S)$, while the other connected components, if any, are contractible.
\end{proposition}

These two propositions justify that $Cut_b(S)_\lambda$ is a good approximation of $Cut_b(S)$, for $\lambda>0$ small enough. Before proving proposition \ref{prop:homotopy type}, we prove the following lemma.
\begin{lemma}\label{lemma:gradient angle}
  Let $x\in Cut_b(S)$ be such that there exists two unit speed minimizing geodesics $\gamma_1, \gamma_2 : [0,d_b(x)] \to S$ such that $\gamma_i (0) = b$ and $\gamma_i (d_b(x)) = x$. Let $\theta \in(0,\pi)$ be the angle between $\gamma_1$ and $\gamma_2$ at $x$. Then, we have
  \begin{equation*}
    \abs{\nabla_{_S} d_b}(x)\leq \cos(\theta/2).
  \end{equation*}
\end{lemma}
\begin{proof}
  For $i = 1,2$, let us set $v_i = -\dot{\gamma_i}(d_b(x))$. Let $v \in T_xS$. Let us denote by $\exp_x$ the Riemannian exponential map at the point $x$. Let $t_0\in(0,d_b(x))$ and $x_i = \exp_{v_i t_0}$. Note that we have $x\notin Cut_{x_i}(S)$, so the function $d_{x_i}$ is smooth at $x$, and its gradient is $-v_i$.
  Given $v\in T_xS$ such that $\abs{v} = 1$, we have
  \begin{align}
    \partial^+_v d_b(x)
    &= \lim\limits_{t\to 0^+} \frac{d_b(\exp_x(vt))-d_b(x)}{t} \nonumber
    \\
    &\leq \lim\limits_{t\to 0^+} \frac{d_{x_i}(\exp_x(vt)) + d_b(x_i) - (d_{x_i}(x) + d_b(x_i))}{t} \nonumber
    \\
    & = \lim\limits_{t\to 0^+} \frac{d_{x_i}(\exp_x(vt)) - d_{x_i}(x)}{t} \nonumber
    \\
    &= -v \cdot v_i \nonumber
  \end{align}
  Given that the angle between $v_1$ and $v_2$ is $\theta$, there exists $i\in\{1,2\}$, such that the angle between $v$ and $v_i$ is at most $\pi - \theta/2$. Thus the last inequality gives $\partial^+_v d_b(x) \leq \cos(\theta/2)$. This concludes the proof.
\end{proof}

Using lemma \ref{lemma:gradient angle}, Proposition \ref{prop:homotopy type} will mainly be a consequence of \cite[Proposition 3.4]{dey2009cut} and the proof of \cite[Proposition 3.5]{dey2009cut}. Following \cite{dey2009cut}, we will use the following terminology.
A point $x$ of a finite graph $G$ is called a tree point if $G\setminus\{x\}$ has a connected component whose closure is a tree. Otherwise, $x$ is called a cycle point. It is a consequence of the proof of \cite[Proposition 3.5]{dey2009cut} that any closed connected subset of $G$ that contains all cycle points is a deformation retract of $G$.
\begin{proof}[Proof of proposition \ref{prop:homotopy type}]
  As $S$ is real analytic, the cut locus $Cut_b(S)$ is a finite graph (see \cite{myers1936connections} in dimension $2$, and \cite{buchner1977simplicial} for the generalization to arbitrary dimensions).
  According to the lemma \ref{lemma:gradient angle}, given any $\theta>0$, if $\lambda$ has been taken small enough, then for any point $x\in Cut_b(S)\setminus Cut_b(S)_\lambda$, the angle between the minimizing geodesics from $b$ to $x$ is smaller than $\theta$.
  Given two unit speed minimizing geodesics $\gamma_1$ and $\gamma_2$, following \cite{dey2009cut}, the \emph{spread} between $\gamma_1$ and $\gamma_2$ is defined as
  \begin{equation*}
    spd(\gamma_1,\gamma_2) = \sup\limits_t d(\gamma_1(t),\gamma_2(t)).
  \end{equation*}
  As geodesics verify a second order differential equation, if their angle at their common starting point is small, then their spread is also small. Therefore, applying \cite[Proposition 3.4]{dey2009cut}, we deduce that if $\lambda$ has been taken small enough, then any point $x\in Cut_b(S)\setminus Cut_b(S)_\lambda$ is a \emph{tree point} of $Cut_b(S)$. It remains to show that $Cut_b(S)_\lambda$ is closed to conclude that it is a deformation retract of $Cut_b(S)$ and conclude the proof. But this is a consequence of the fact that $d_b$ is semiconcave, and the upper semicontinuity of the generalized gradient of convex functions.
\end{proof}
Therefore, we will use $Cut_b(S)_\lambda$ as an approximation of $Cut_b(S)$ for $\lambda$ small enough.


\section{Approximation with a variational problem}\label{sec:variational problem}
For $m>0$, recall that $u_m$ is the minimizer in \eqref{eq:gradient constraint}. For $\lambda>0$, let us define the set $E_{m,\lambda}$ by
\begin{equation*}
  E_{m,\lambda} := \left\{x\in S \setminus\{b\} : \abs{\nabla_{_S} u_m(x)}^2 \leq 1 - \frac{\lambda^2}{u_m^2(x)}\right\}.
\end{equation*}
We have the following theorem (see \cite[Theorem 1.1 and Theorem 1.3]{generauCutLocusCompact2020}):
\begin{theorem}\label{thm:big theorem}
  For any $m>0$, the function $u_m$ is locally $C^{1,1}$ on $S\setminus \{b\}$. For any $m>m'>m_0$,
  \begin{equation}\label{eq:double inclusion}
    Cut_b(S) \subset \{\abs{\nabla_{_S} u_m} < 1 \} \subset \{\abs{\nabla_{_S} u_{m'}} < 1 \}.
  \end{equation}
  Moreover,
  \begin{equation}\label{eq:convergence elastic set}
     \{\abs{\nabla_{_S} u_m} < 1 \}\tend{m}{+\infty}Cut_b(S) \quad \text{in the Hausdorff sense.}
  \end{equation}
  Finally, for any $\varepsilon>0$,
  \begin{equation}
    \sup\limits_{x \in E_{m,\lambda}} d(x,Cut_b(S)_{\lambda}) \tend{m}{+\infty}0, \label{eq:convergence lambda cut locus}
    \quad \text{and} \quad  \sup\limits_{x \in Cut_b(S)_{\lambda+\varepsilon}} d(x,E_{m,\lambda}) \tend{m}{+\infty}0.
  \end{equation}
\end{theorem}

Therefore, we can use $E_{m,\lambda}$ as an approximation of $Cut_b(S)_\lambda$. All in all, we will use $E_{m,\lambda}$ as an approximation of $Cut_b(S)$.


\section{Discretization}\label{sec:discretization}

\subsection{Finite elements of order $r$ on a surface approximation of order $k$}

In this section we introduce a discretization framework adapted to variational problem
\eqref{eq:gradient constraint} based on finite elements. We follow the notations
 of \cite{demlow2009higher, dziuk2013finite}.

Let $S$ be a compact oriented smooth two-dimensional surface embedded in $\R^3$.
For $x \in S$, we denote by $\nu (x)$ the oriented normal vector field on $S$.
Let $d:\R^3 \rightarrow \R$ be the signed distance associated to $S$ and
$U_\eta = \{x\in \R^3, \, |d(x)| < \eta \} $ the tubular
neighborhood of $S$ of width $\eta >0$. It is well known that if $\eta$ is
small enough (for  instance
$0< \eta < \min_{i = 1,2} {\frac{1}{|\kappa_i|}}_{L_{\infty}(S)}$ where the $(\kappa_i)$
stand for the extremal sectional curvatures of S), then for every $x \in U_\eta$
it exists a unique $a(x) \in S$ such that

\begin{equation}
  \label{eq:proj}
  x = a(x) + d(x) \nu(a(x)) = a(x) + d(x) \nabla d(x).
\end{equation}
We consider $S_h^1$ a triangular approximation of $S$ whose vertices lie on $S$
and whose faces are quasi-uniform and shape regular of  diameter at most $h>0$.
Moreover, we will assume that $\calTT_h$, the set of triangular faces of $S_h$,
are contained in some tubular neighborhood $U_\eta$ such that  the map $a$
defined by  \eqref{eq:proj} is unique.

For $k \geq 2$ and for a triangle $T \in \calTT_h$, we consider
the $n_k$ Lagrange basis functions $\Phi_1^k,\dots \Phi_{n_k}^k $ of degree $k$
and define the discrete projection on $S_h$ by:

\begin{equation}
  \label{eq:projk}
  a_k(x) = \sum_{j=1}^{n_k} a(x_j)  \Phi_{j}^k(x)
\end{equation}
where $x_1,\dots,x_{n_k}$ are the nodal points associated to the basis functions.
Now we can define $S_h^k$ a polynomial approximation of order $k$ of $S$ associated
 to $\calTT_h$

\begin{equation}
  \label{eq:Spoly}
  S_h^k = \{ a_k(x),\, x\in S_h \}.
\end{equation}
Observe that by definition the image by $a$ of the nodal points are both on $S$
and on  $S_h^k$. Let us now introduce the finite element spaces on $S_h = S_h^1$
and $S_h^k$ for $k \geq 2$. For every integer $r\geq 1$, let

\begin{equation}
  L^r_h  = \{ \chi  \in \calC^0(S_h),\, \chi|_{T} \in \PP_r,
  \forall  T \in \calTT_h\}
\end{equation}
where $\PP_r$ is the family of polynomials of degree at most $r$. Analogously,
for $k \geq 2$ let

\begin{equation}
  L^{r,k}_h  = \{ \hat{\chi}  \in \calC^0(S_h^k),\, \hat{\chi} =
  \chi \circ  a_k^{-1}, \, \text{for some } \,  \chi \in L^r_h \}.
\end{equation}
Analogously to \eqref{eq:gradient constraint}, we define
\begin{equation} \label{eq:gradient constraint_h}
  \min_{\substack{u\in L_h^{r,k} \\ \abs{\nabla_{_{S_h^k}} u} \leq 1 \\u(b)=0}}
                   F_h^k(u)
\end{equation}
where $ F_h^k(u) = \int_{S_h^k} \left( \abs{\nabla_{_{S_h^k}} u}^2 - mu\right)$ and $b$
 some fixed nodal points of the mesh $\calTT_h$.

\subsection{Convergence of the lifted minimizers}
In order to prove the convergence of our numerical approach, let us first establish
that our discrete problem converges in values in the sense of proposition \ref{prop:values}.
For a function $u$ defined on $S_h^k$, we introduce its lifted  function $u^l$
onto $S$ defined by the relation $u^l(b) = u(x)$ for $b \in S$ where $x$
is the unique point  of $S_h^k$ which satisfies $a(x) = b$.

Below, we focus our analysis in the piecewise linear case $r = k = 1$ which contains
 all the  main ingredients of a proof for the general $(r,k)$ case.
 For every $h>0$, the convex  optimization problem \eqref{eq:gradient constraint_h}
 has a unique solution.
\begin{lemma}\label{lemma:differential of a}
  The differential of the projection $a$ onto $S$, when restricted to the tangent space of $S_h$, is the identity, up to order $2$ in $h$:
  \begin{equation*}
    Da_{|_{TS_h}} = Id + \mathcal{O}(h^2).
  \end{equation*}
  The second differential of $a$, when restricted to the tangent space of $S_h$, is null, up to order $1$ in $h$:
  \begin{equation*}
    D^2a_{|_{TS_h}} =  \mathcal{O}(h).
  \end{equation*}
\end{lemma}
\begin{proof}
  The identity estimate on $Da$ is a direct consequence of \cite[equations (4.12), (4.13) and (4.11)]{dziuk2013finite}, and the fact that, following the notations of \cite[lemma 4.1]{dziuk2013finite}, we have $\nu_{n+1}^2 = 1 - \sum\limits_{j\leq n }\nu_j^2$.
  The estimate on $D^2a$ follows from the same equations, plus the identity $D^2a(x) = -2\nabla d(x)D^2d(x)$.

\end{proof}

  Defining $ F(u) = \int_{S} \left( \abs{\nabla_{_S} u}^2 - mu\right)$, we have

\begin{proposition}
  \label{prop:values}
  Let $u_{m,h}$ be the
  solution of problem \eqref{eq:gradient constraint_h} for $k=r=1$. Let $Lu_{m,h}^l := \frac{u_{m,h}^l}{|\nabla_{_{S}} u_{m,h}^l|_{L_\infty(S)}}$ be the $1$-Lipschitz normalization of $u_{m,h}^l$. Then,  $Lu_{m,h}^l \in H^1(S)$ and
  $$F(Lu_{m,h}^l)  = \min_{\substack{u\in H^1(S)\\ \abs{\nabla_{_{S}} u} \leq 1 \\u(b)=0}}
  F(u)  + \mathcal{O}(h^{\frac{1}{2}}).$$
\end{proposition}

\begin{proof}
    \emph{step 1.} Let $u_m$ be the solution of problem \eqref{eq:gradient constraint}. For $\varepsilon>0$, let $u_{m,\varepsilon}:S\to\R$ be defined by:
    \begin{equation*}
      u_{m,\varepsilon} =
      \begin{cases}
        \frac{d_b(x)^2}{2\varepsilon} \quad &\text{if} \quad d_b(x) \leq \varepsilon
        \\
        u_m(x) - \frac{\varepsilon}{2} \quad &\text{if} \quad d_b(x) \geq \varepsilon.
      \end{cases}
    \end{equation*}
    According to \cite[Lemma 3.3]{generauCutLocusCompact2020}, we have $u_m = d_b$ in a neighborhood of $b$. Therefore, for $\varepsilon>0$ small enough, we have $u_m = d_b$ on $B(b,2\varepsilon)$. In particular, we deduce that $u_{m,\varepsilon}$ is $C^1$ on $S$.
    As $d_b^2$ is smooth in a neighborhood of $b$, the gradient of $d_b^2/2\varepsilon$ is $\mathcal{O}(\varepsilon^{-1})$-Lipschitz on $B(b,\varepsilon)$. Moreover, as $u_m = d_b$ on $B(b,2\varepsilon)$, the gradient of $u_m$ is $\mathcal{O}(\varepsilon^{-1})$-Lipschitz on $B(b,2\varepsilon)\setminus B(b,\varepsilon)$.
    According to lemma \cite[Proposition 3.4]{generauCutLocusCompact2020}, $u_m$ is also locally $C^{1,1}$ on $S\setminus \{b\}$. Therefore its gradient is $\mathcal{O}(\varepsilon^{-1})$-Lipschitz on $S\setminus B(b,\varepsilon)$.
    All in all, we obtain that $u_{m,\varepsilon}$ is $C^{1,1}$ on $S$, and the Lipschitz constant of its gradient is $\mathcal{O}(\varepsilon^{-1})$. Furthermore, as $d_b$ and $u_m$ are both $1$-Lipschitz, we have $\abs{\nabla u_{m,\varepsilon}}\leq1$. Now for $\varepsilon>0$, consider
    \begin{equation*}
      v_{h,\varepsilon} := \frac{I_hu_{m,\varepsilon}}{|\nabla_{_{S_h}} I_hu_{m,\varepsilon}|_{L_\infty(S_h)}},
    \end{equation*}
    where $I_hu_{m,\varepsilon}$ is the $\PP^1$ Lagrange interpolation of $u_{m,\varepsilon}$ on $S_h$. For $x\in S_h$,
    observe that we have the relation $I_h u_{m,\varepsilon}(x) = I_h (u_{m,\varepsilon} \circ a)(x)$ which
    says that $I_h u_{m,\varepsilon}$ is the standard (flat) interpolation of the composed function
    $u_{m,\varepsilon} \circ a$. From lemma \ref{lemma:differential of a}, we know that on every triangle, the differential of $a$ is $\mathcal{O}(h)$-Lipschitz, and $a$ is $\mathcal{O}(1)$-Lipschitz.
    As the gradient of $u_{m,\varepsilon}$ is $\mathcal{O}(\varepsilon^{-1})$-Lipschitz, we deduce that on every triangle, the gradient of $u_m\circ a$ is $\mathcal{O}(\varepsilon^{-1})$-Lipschitz.
    By the quasi uniformity of the mesh, we obtain the uniform interpolation estimates on $S_{h}$:
    \begin{equation}\label{eq:interpolation estimate}
      I_h u_{m,\varepsilon}(x) = (u_{m,\varepsilon} \circ a)(x) + \mathcal{O}(\varepsilon^{-1}h^2)
    \end{equation}
    and
    \begin{equation*}
      \nabla_{_{S_h}} I_h u_{m,\varepsilon}(x) = \nabla_{_{S_h}} (u_{m,\varepsilon} \circ a)(x) + \mathcal{O}(\varepsilon^{-1}h).
    \end{equation*}
    With lemma \ref{lemma:differential of a}, we deduce for all $x\in S_{h}$,
    \begin{equation}\label{eq:interpolation gradient estimate}
      \nabla_{_{S_h}}  I_h u_{m,\varepsilon}(x) = \nabla_{_{S}} u_{m,\varepsilon}(a(x)) + \mathcal{O}(\varepsilon^{-1}h).
    \end{equation}
    Recall that we have $|\nabla_{_{S}} u_{m,\varepsilon}|_{L_\infty(S)} = 1$. Therefore the last identity yields
    $$|\nabla_{_{S_h}}  I_h u_{m,\varepsilon}|_{L_\infty(S_{h,\varepsilon})} = 1 +  \mathcal{O}(\varepsilon^{-1}h).$$
    Thus, $v_{h,\varepsilon} = I_h u_{m,\varepsilon}(1 + \mathcal{O}(\varepsilon^{-1}h))$, and so
    \begin{equation}\label{eq:F_h intermediate 1}
      F_h(v_{h,\varepsilon}) = F_h(I_h u_{m,\varepsilon}) +  \mathcal{O}(\varepsilon^{-1}h).
    \end{equation}
    Applying lemma \ref{lemma:differential of a} again, with a simple change of variable, we find that for any function $f:S_h \to \R$,
    \begin{equation*}
      \int_{S_h}f \circ a = \int_{S} f + \mathcal{O}(h^2).
    \end{equation*}
    Recalling \eqref{eq:interpolation estimate} and \eqref{eq:interpolation gradient estimate}, we obtain
    \begin{equation}\label{eq:F_h intermediate 2}
      F_h(I_h u_{m,\varepsilon}) = F(u_{m,\varepsilon}) +  \mathcal{O}(\varepsilon^{-1}h).
    \end{equation}
    Furthermore, we have
    \begin{equation*}
      \int_S\abs{u_{m,\varepsilon}-u_m} \leq \mathcal{O}(\varepsilon)\quad \text{and} \int_S \abs{\nabla u_{m,\varepsilon} - \nabla u_m}^2 \leq \mathcal{O}(\varepsilon^2),
    \end{equation*}
    so
    \begin{equation*}
      F(u_{m,\varepsilon}) = F(u_m) + \mathcal{O}(\varepsilon).
    \end{equation*}
    Combining this with \eqref{eq:F_h intermediate 1} and \eqref{eq:F_h intermediate 2}, we find
    \begin{equation*}
      F_h(v_{h,\varepsilon}) = F(u_m) + \mathcal{O}(\varepsilon^{-1}h) + \mathcal{O}(\varepsilon).
    \end{equation*}
    Choosing $\varepsilon = h^{\frac{1}{2}}$, this yields
    \begin{equation}\label{eq:min comparison}
      \min_{\substack{u\in H^1(S_h)\\ \abs{\nabla_{_{S_h}} u} \leq 1 \\u(b)=0}} F_h \leq \min_{\substack{u\in H^1(S)\\ \abs{\nabla_{_{S}} u} \leq 1 \\u(b)=0}} F +  \mathcal{O}(h^{\frac{1}{2}}).
    \end{equation}

    \emph{step 2.} Symmetrically, let $u_{m,h}$ the solution of the discrete
    problem  \eqref{eq:gradient constraint_h}, $u^l_h := u_{m,h} \circ (a_{|_{S_h}})^{-1}$ its lifted version on $S$, and $Lu_{m,h}^l :=  \frac{u^l_h}{{|\nabla_{_{S_h}} u^l_h|}_{L_\infty (S_h)}} $. We show as before, using the equation $u_{m,h} = u^l_h \circ a$, that $F(Lu_{m,h}^l) = F_h(u_{m,h}) +  \mathcal{O}(h)$.
    With \eqref{eq:min comparison}, this implies
    \begin{equation*}
      \min_{\substack{u\in H^1(S)\\ \abs{\nabla_{_{S}} u} \leq 1 \\u(b)=0}} F \leq F(Lu_{m,h}^l) \leq \min_{\substack{u\in H^1(S)\\ \abs{\nabla_{_{S}} u} \leq 1 \\u(b)=0}} F +  \mathcal{O}(h^{\frac{1}{2}}),
    \end{equation*}
    which concludes the proof of the proposition.
\end{proof}

We can now establish the convergence of the minimizers:
\begin{proposition}\label{prop:convergence minimizers}
  \begin{equation*}
    \abs{\nabla u_{m,h}^l - \nabla u_m}^2_{L^2(S)} = \mathcal{O}(h^{\frac{1}{2}}) \quad \text{and} \quad \abs{u_{m,h}^l - u_m}_{L^1(S)} = \mathcal{O}(h^{\frac{1}{2}}).
  \end{equation*}
\end{proposition}

\begin{proof} Consider $v = \frac12 (Lu_{m,h}^l + u_m)$. Then, $v$ is admissible for
   problem \eqref{eq:gradient constraint}, so $F(v)\geq F(u_m)$.
    Moreover, the following algebraic identity holds
$$
F(v) = \frac12 F(Lu_{m,h}^l) + \frac12 F(u_m) - \frac14 \int_S |\nabla_{_{S}} u_m - \nabla_{_{S}}  Lu_{m,h}^l|^2.
$$
Therefore, we have
$$
\frac12 F(Lu_{m,h}^l) - \frac12 F(u_m) \geq \frac14 \int_S |\nabla_{_{S}} u_m - \nabla_{_{S}} Lu_{m,h}^l|^2,
$$
which proves, with proposition \ref{prop:values}, that
\begin{equation}\label{eq:l2 gradient}
  \abs{\nabla Lu_{m,h}^l - \nabla u_m}^2_{L^2(S)}  = \mathcal{O}(h^{\frac{1}{2}}).
\end{equation}
Moreover, we have
\begin{equation*}
   F(Lu_{m,h}^l) -  F(u_m) = \int_S \left(\abs{\nabla Lu_{m,h}^l}^2 - \abs{\nabla u_{m}}^2\right) - m \int_S \left(Lu_{m,h}^l - u_m\right).
\end{equation*}
The last two equations imply
\begin{equation}\label{eq:l1 function}
  \abs{ Lu_{m,h}^l -  u_m}_{L^1(S)}  = \mathcal{O}(h^{\frac{1}{2}}).
\end{equation}
As in the proof of proposition \ref{prop:values}, using the relation $u_{m,h} = u_{m,h}^l \circ h$, we show that $Lu_{m,h}^l = u_{m,h}^l(1 + \mathcal{O}(h^2))$. Together with \eqref{eq:l2 gradient} and \eqref{eq:l1 function}, this concludes the proof.
\end{proof}
We just proved that the sequence of the lifted minimizers converges with an order
 at least $1/2$ to the minimizer of problem  \eqref{eq:gradient constraint}. By analogy with the more standard variational context \cite{demlow2009higher, dziuk2013finite},
  we expect a convergence of order $\mathcal{O}((h^r + h^{k+1})^{\frac{1}{2}})$
 using an approximation of orders $(r,k)$.

\subsection{Convergence in measure to the elastic set}
Let us recall that the set $E_{m,\lambda}$ is defined by
\begin{equation*}
  E_{m,\lambda}=\left\{x\in S \setminus\{b\} : \abs{\nabla_{_S} u_m(x)}^2 \leq 1 - \frac{\lambda^2}{u_m^2(x)}\right\}.
\end{equation*}

\begin{proposition}
For any $\lambda >0$ and $\varepsilon>0$ with $\varepsilon<\lambda/2$, let us define
\begin{equation*}
  E_{m,\lambda,h} := \left\{x\in S \setminus\{b\} : \abs{\nabla_{_S} u_{m,h}^l(x)}^2 \leq 1 - \frac{\lambda^2}{(u_{m,h}^l)^2(x)}\right\}.
\end{equation*}
Then, we have
\begin{equation*}
  \abs{E_{m,\lambda+\varepsilon}\setminus E_{m,\lambda,h}} = \mathcal{O}(h^{\frac{1}{2}}) \quad \text{and} \quad \abs{E_{m,\lambda,h}\setminus E_{m,\lambda-\varepsilon}} = \mathcal{O}(h^{\frac{1}{2}}).
\end{equation*}
\end{proposition}

\begin{proof}
By definition of $E_{m,\lambda}$ and $E_{m,\lambda,h}$, we have
\begin{equation*}
  E_{m,\lambda+\varepsilon}\setminus E_{m,\lambda,h} \subset \left\{\abs{\nabla u_{m,h}^l}^2 - \abs{\nabla u_m}^2 > \frac{(\lambda+\varepsilon)^2}{u_m^2} - \frac{\lambda^2}{(u_{m,h}^l)^2} \right\}.
\end{equation*}
Therefore, on $E_{m,\lambda+\varepsilon}\setminus E_{m,\lambda,h}$, we have
\begin{align*}
  \abs{\nabla u_{m,h}^l}^2 - \abs{\nabla u_m}^2
  & > \frac{(\lambda+\varepsilon)^2-\lambda^2}{u_m^2} + \lambda^2 (\frac{1}{u_m^2} - \frac{1}{(u_{m,h}^l)^2})
  \\
  & \geq \frac{2 \varepsilon \lambda + \varepsilon^2}{u_m^2} - \lambda^2 \frac{2}{\min(u_m,u_{m,h}^l)^3}\abs{u_m - u_{m,h}^l}
  \\
  & = \frac{2 \varepsilon \lambda + \varepsilon^2}{(\diam S)^2} - \lambda^2 \frac{2}{(u_m + \mathcal{O}(h^{\frac{1}{2}}))^3}\abs{u_m - u_{m,h}^l},
\end{align*}
where $\diam S$ is the diameter of $S$.
By definition of $E_{m,\lambda}$, we also have $E_{m,\lambda+\varepsilon} \subset \{u_m\geq (\lambda+\varepsilon)\}$, so on $E_{m,\lambda+\varepsilon}\setminus E_{m,\lambda,h}$,
\begin{align*}
  \abs{\nabla u_{m,h}^l}^2 - \abs{\nabla u_m}^2
  > \frac{2 \varepsilon \lambda + \varepsilon^2}{(\diam S)^2} - \lambda^2 \frac{2}{(\lambda + \varepsilon + \mathcal{O}(h^{\frac{1}{2}}))^3}\abs{u_m - u_{m,h}^l}.
\end{align*}
So for $h$ big enough, we have
\begin{align*}
  &E_{m,\lambda+\varepsilon}\setminus E_{m,\lambda,h}
  \\
  &\quad \subset \left\{\abs{\nabla u_{m,h}^l}^2 - \abs{\nabla u_m}^2 + 2 \lambda \abs{u_m - u_{m,h}^l} > \frac{2 \varepsilon \lambda + \varepsilon^2}{(\diam S)^2} \right\}.
  \\
  &\quad \subset \left\{\abs{\nabla u_{m,h}^l}^2 - \abs{\nabla u_m}^2 > \frac{2 \varepsilon \lambda + \varepsilon^2}{2(\diam S)^2} \right\} \cup \left\{\abs{u_m - u_{m,h}^l} > \frac{2 \varepsilon \lambda + \varepsilon^2}{4 \lambda (\diam S)^2} \right\}.
\end{align*}
Now from proposition \ref{prop:convergence minimizers}, we know that for any $\eta>0$, we have the following estimates
\begin{equation*}
  \eta\abs{\left\{\abs{\abs{\nabla u_{m,h}^l}^2 - \abs{\nabla u_m}^2} > \eta \right\}}\leq \int_S \abs{\abs{\nabla u_{m,h}^l}^2 - \abs{\nabla u_m}^2} = \mathcal{O}(h^{\frac{1}{2}}),
\end{equation*}
and
\begin{equation*}
  \eta\abs{\left\{\abs{u_m - u_{m,h}^l} > \eta \right\}}\leq \int_S \abs{u_m - u_{m,h}^l} = \mathcal{O}(h^{\frac{1}{2}}).
\end{equation*}
This gives the estimate $\abs{E_{m,\lambda+\varepsilon}\setminus E_{m,\lambda,h}} = \mathcal{O}(h^{\frac{1}{2}})$. The other estimate is proved by the same method.
\end{proof}

\begin{remark}
  We expect a convergence of order $\mathcal{O}((h^r + h^{k+1})^{\frac{1}{2}})$
 using an approximation of orders $(r,k)$.
\end{remark}

Sections \ref{sec:lambda cut locus}, \ref{sec:variational problem} and \ref{sec:discretization} together justify that the set $E_{m,\lambda,h}$ is a good approximation of the cut locus of $b$ in $M$, if $m$ is big enough, and $\lambda$ and $h$ are small enough.


\section{Numerical illustrations}\label{sec:numerical}

\subsection{Cut locus approximation}

We established the convergence of the minimizers of solutions of problems
\eqref{eq:gradient constraint} when $h$ tends to $0$. For a fixed $h>0$,
this convex discrete problems is of quadratic type with an infinite
 number of conic pointwise constraints.
By the way, it is important to observe that  for $k=r=1$, the gradient pointwise
bounds for a function of $\PP^1$ is
equivalent to a single discrete conic constraint on every triangle with respect
to the degrees of freedom of  $\PP^1(\calTT_h)$.

Nevertheless, we observed in our experiments that using  $\PP^1$ elements may lead to
approximated cut loci with some tiny artificial connected components. Motivated by
 this lack of precision, we use in all following illustrations elements of order
  $r >1$.


 For the general case $r>1$, the bound constraint on the gradient can not be
 easily reduced to a finite set of discrete constraints. In  our experiments, we
approximated the constraint
$| \nabla_{_{S_h^k}} u|_{L^\infty (S_h^k)} \leq 1$  by forcing this constraint only on a
finite number of points of the mesh. In practice, we imposed these constraints
on the Gauss quadrature points of order $g$ on every triangle of $\calTT_h$.

We illustrate in figures \ref{fig:f1}, \ref{fig:f2}, \ref{fig:f3} and \ref{fig:f4}
the approximation of  the cut locus provided by our approach. These computations have been
carried out on meshes of approximated $10^5$ triangles for $k = 2$ and $r = 3$
 using high precision quadrature formula associated to $17$ Gauss points on every element
  of the mesh. Moreover, for $r=3$, we imposed the conic gradient constraints on
  the $g = 9$ Gauss points of every triangle.  In order to solve the resulting
  linear conic constrained quadratic optimization problem,
   we used the \emph{JuMP} modeling language and the finite elements library \emph{Getfem++  }
   \cite{JuMP.jl-2017, renard2006getfem++} combined with \emph{Mosek} optimization
   solver \cite{andersen2000mosek}. For such a precision,  the optimization
   solver identified a solution in less than one hour on a standard computer.


\begin{figure}[htbp]
  \centering
  \begin{tabular}{c c c}
    \includegraphics[width=\widthfig cm]{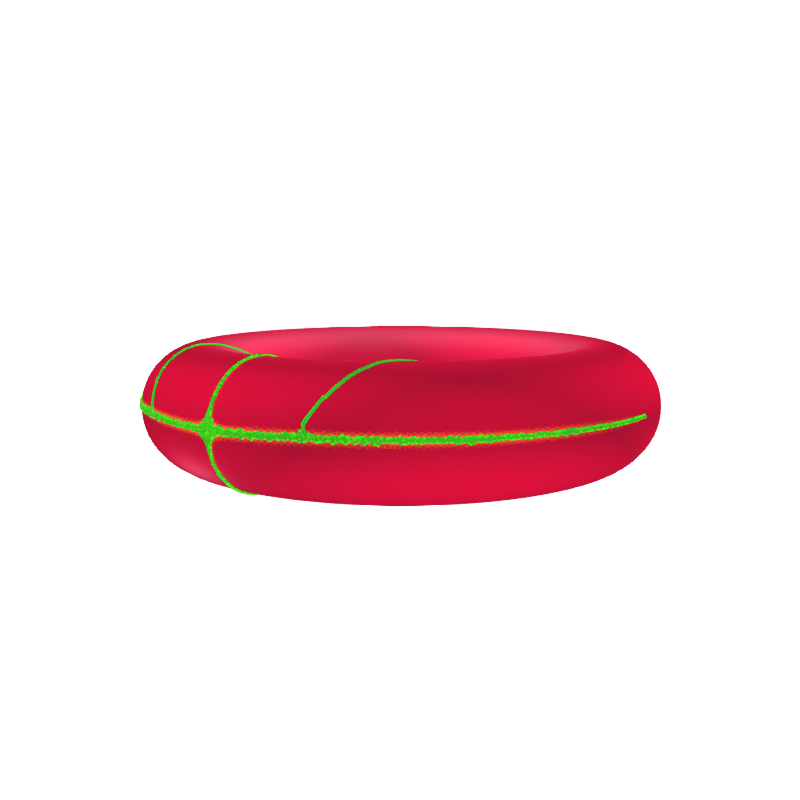}
    & \includegraphics[width=\widthfig cm]{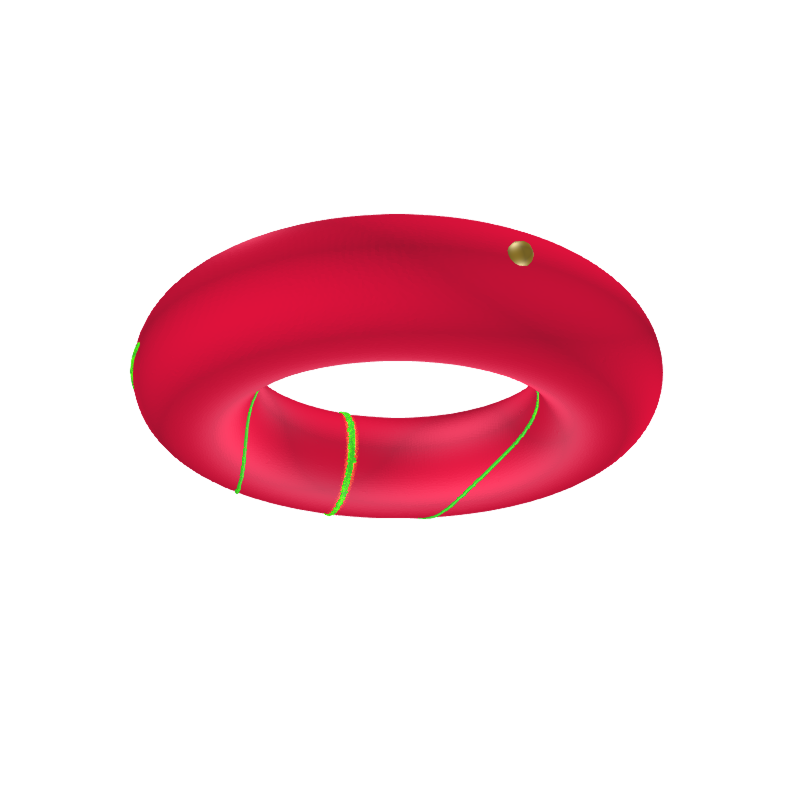}
    & \includegraphics[width=\widthfig cm]{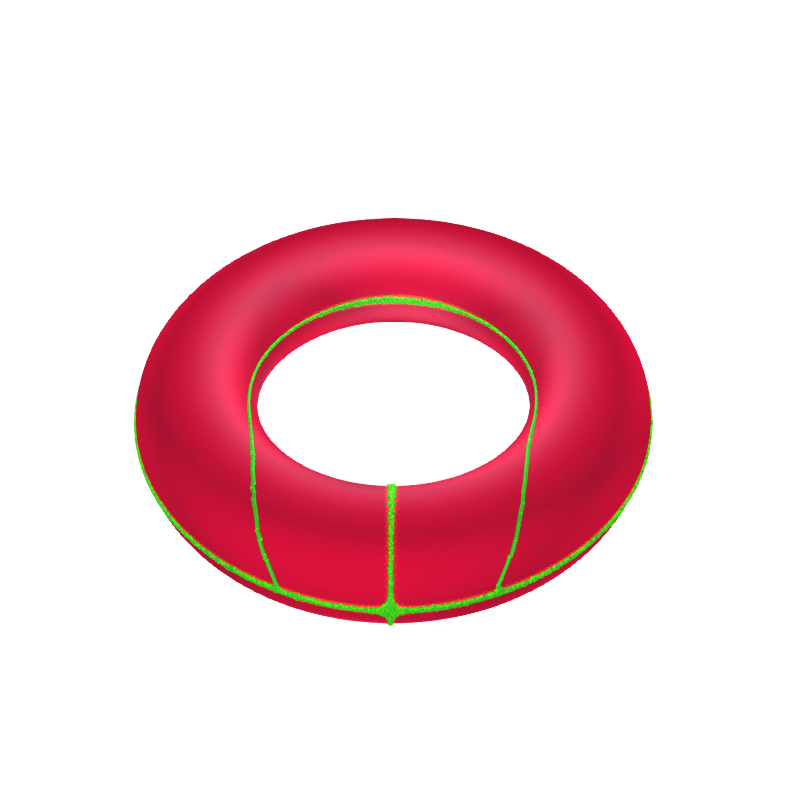}
  \end{tabular}
  \caption{Three different views of the approximation of a cut locus on a standard torus}
  \label{fig:f1}
\end{figure}

\begin{figure}[htbp]
  \begin{tabular}{c c c}
    \includegraphics[width=\widthfig cm]{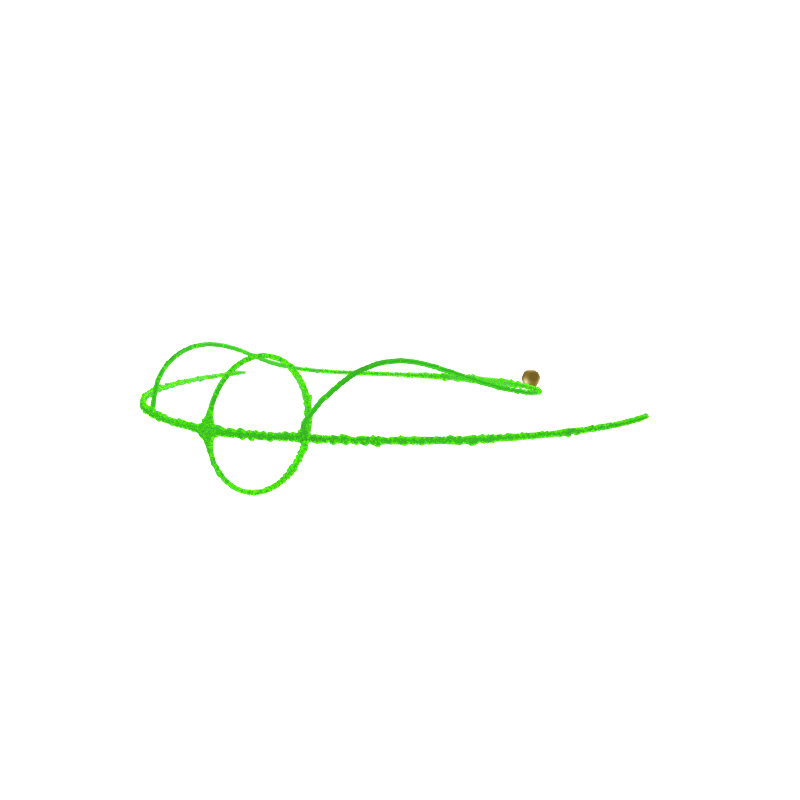}
    & \includegraphics[width=\widthfig cm]{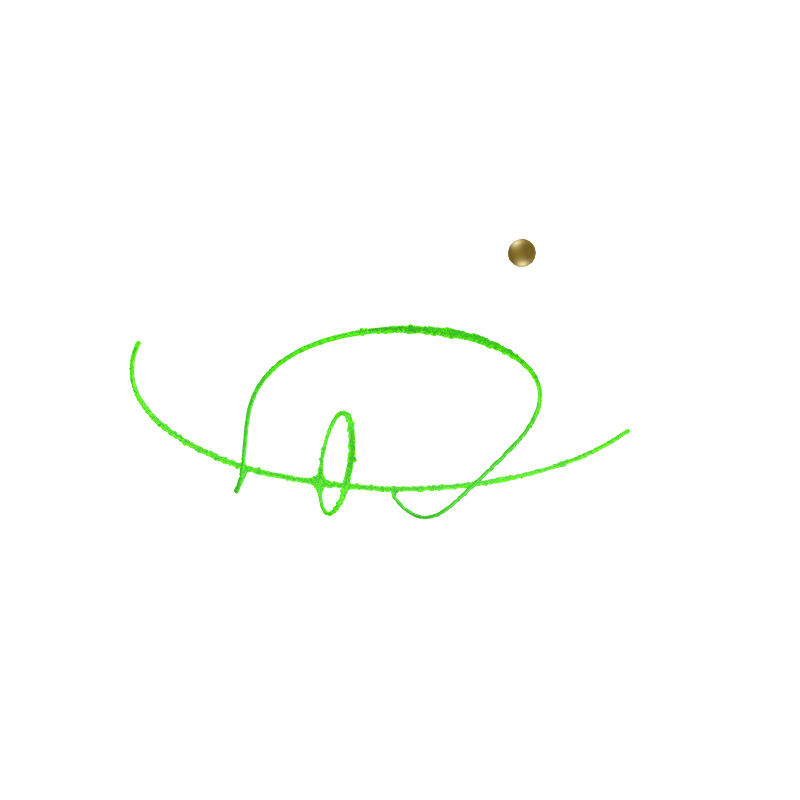}
    & \includegraphics[width=\widthfig cm]{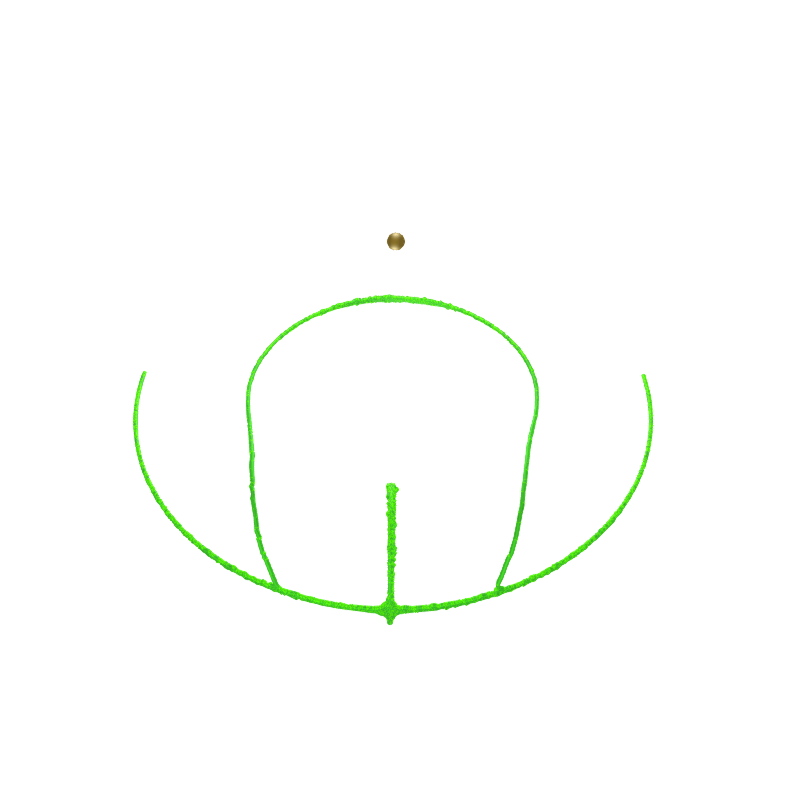}
  \end{tabular}
  \caption{Three different views of the approximation of a cut locus on a standard torus,
           without representing the surface}
  \label{fig:f2}
\end{figure}

\begin{figure}[htbp]
  \centering
  \begin{tabular}{ccc}
    \includegraphics[width=\widthfig cm]{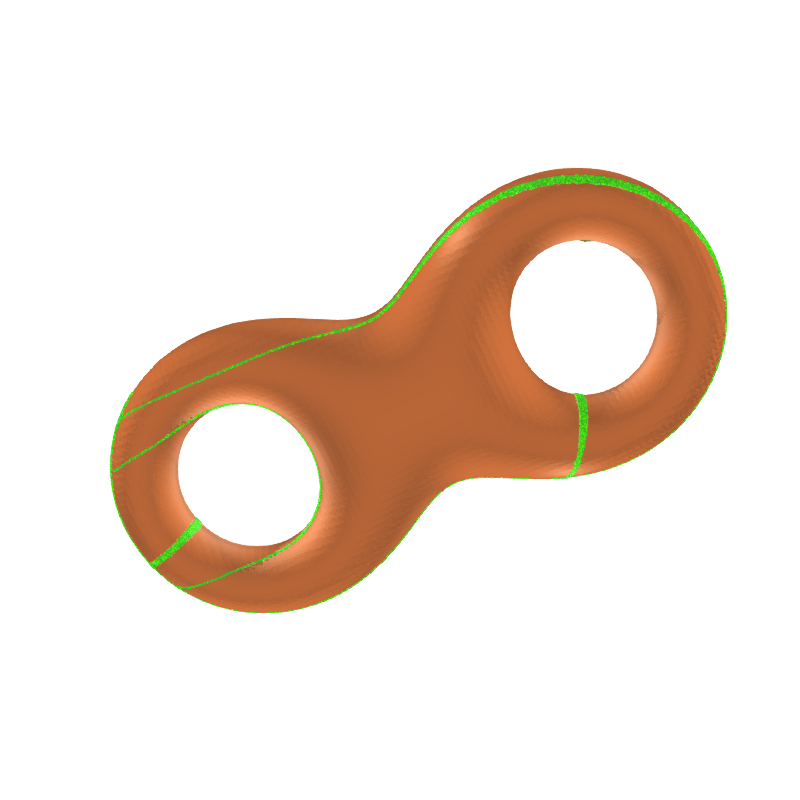}
    & \includegraphics[width=\widthfig cm]{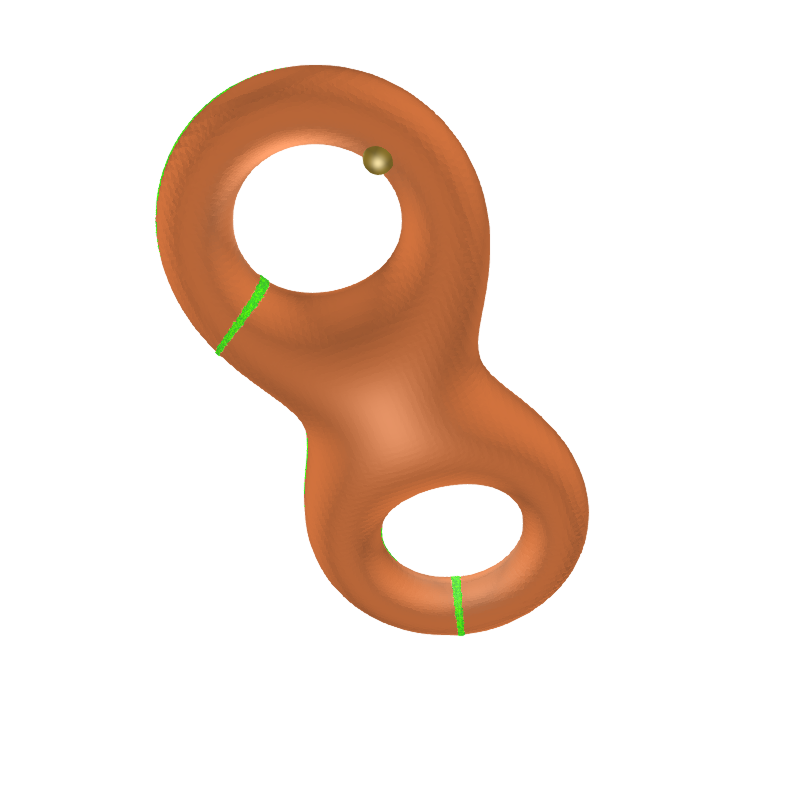}
    & \includegraphics[width=\widthfig cm]{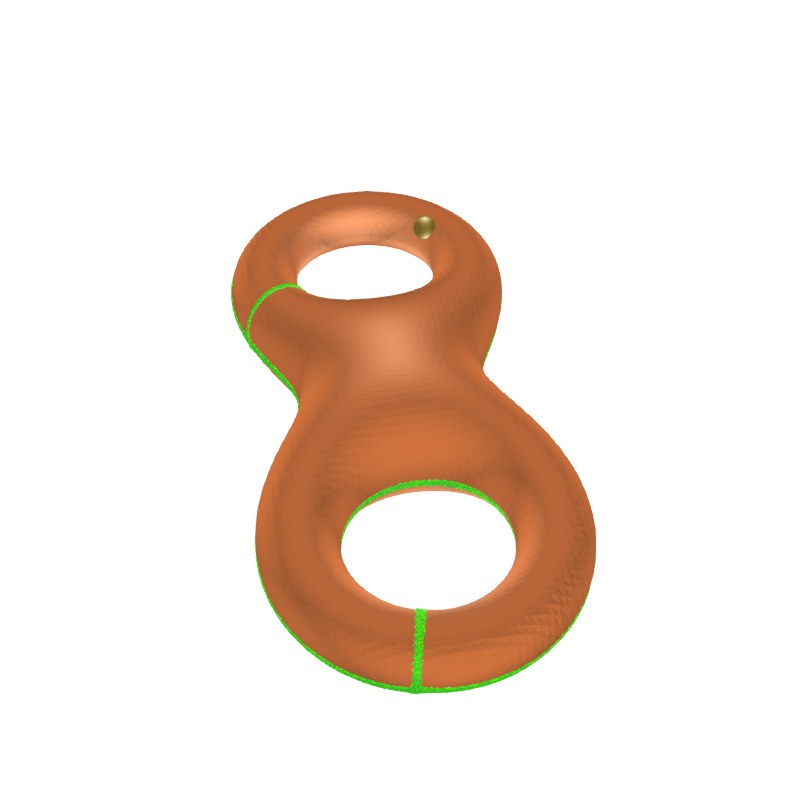} \\
  \end{tabular}
  \caption{Three different views of the approximation of a cut locus on a torus of genus $2$}
  \label{fig:f3}
\end{figure}

\begin{figure}[htbp]
  \begin{tabular}{ccc}
    \includegraphics[width=\widthfig cm]{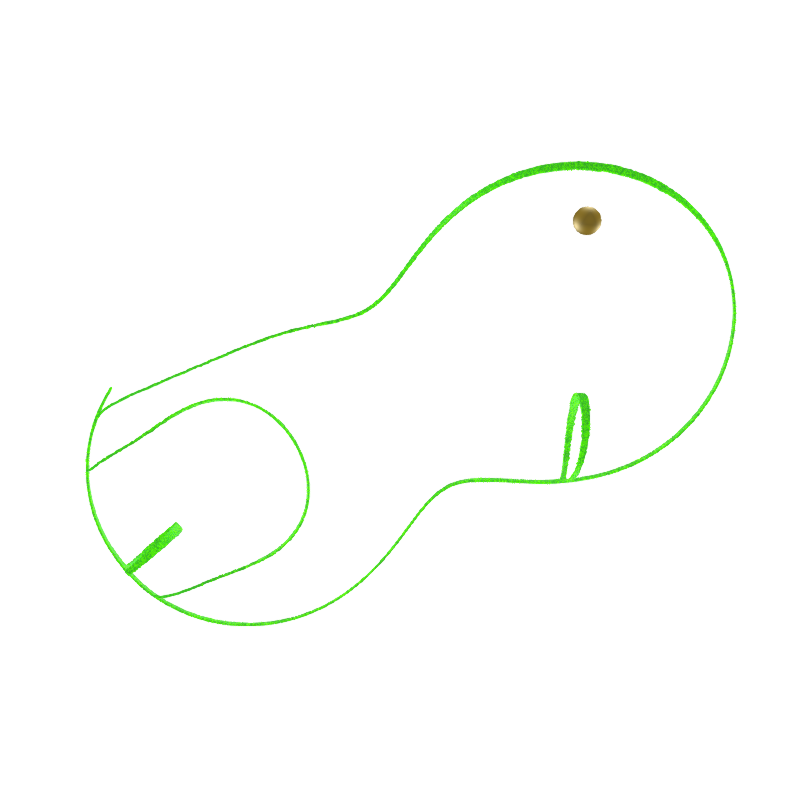}
    & \includegraphics[width=\widthfig cm]{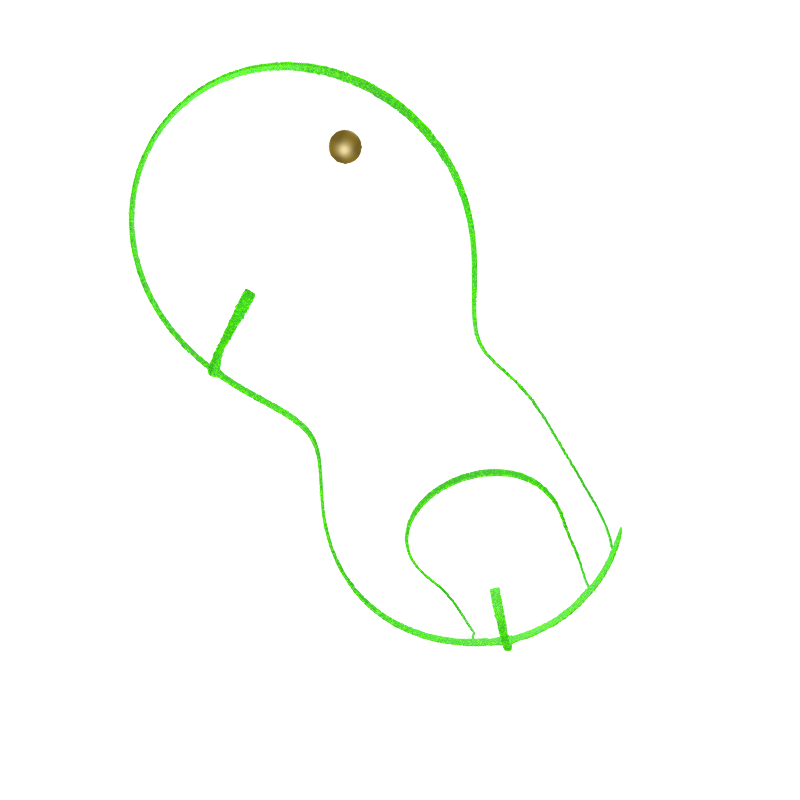}
    & \includegraphics[width=\widthfig cm]{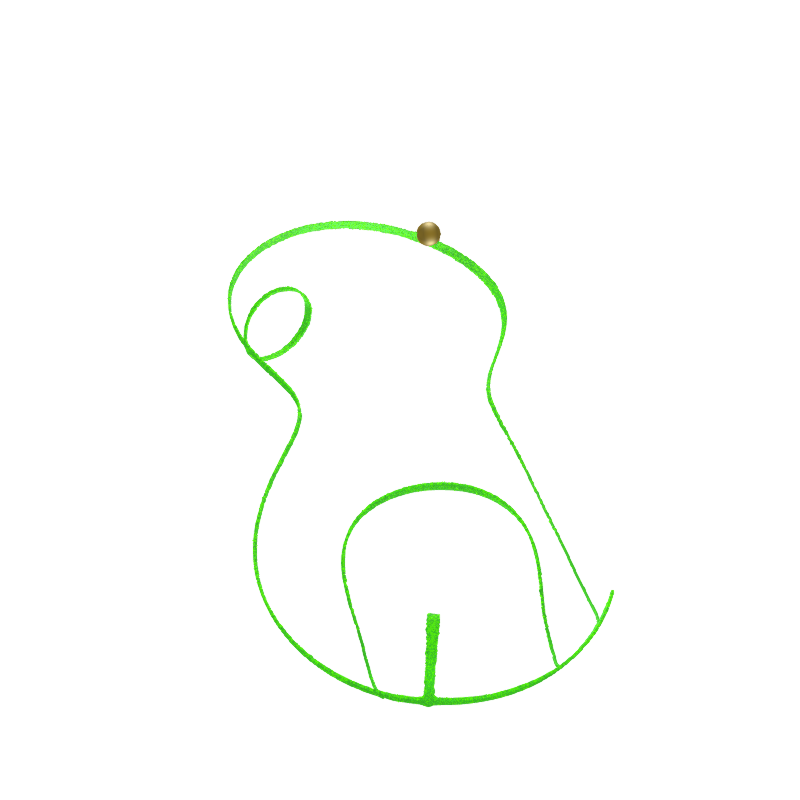}
  \end{tabular}
  \caption{Three different views of the approximation of a cut locus on a torus of genus $2$,
           without representing the surface}
  \label{fig:f4}
\end{figure}
\FloatBarrier

\subsection{Approximation of the boundary of Voronoi cells}

All previous theoretical results still hold if we replace the source point $b$ by any
compact subset of the surface $S$. For instance, if $b$ is replaced by a set of points,
the singular set of the distance function can be decomposed as the union of the
boundary of voronoi cells and the cut loci of every point intersected with its
voronoi cell. As a consequence, if the distribution of source points is homogeneous enough, that is every voronoi cell is small enough, the singular part
of the distance function will be exactly equal to the boundary of the voronoi cells.
We illustrate this remark in the following experiments. We used exactly the
same framework as in previous sections and just replaced the pointwise condition
 at $b$ with the analogous pointwise Dirichlet conditions at every source point.
Figure \ref{fig:vortorus1} and \ref{fig:vortorus2} represent the voronoi
 diagrams obtained with $10$, $30$ and $100$ points for surfaces of genus $2$ and $3$.
  The expected computational complexity is exactly of the same order as with
   a single source point.

\begin{figure}[htbp]
  \begin{tabular}{ccc}
    \includegraphics[width=\widthvoronoi cm]{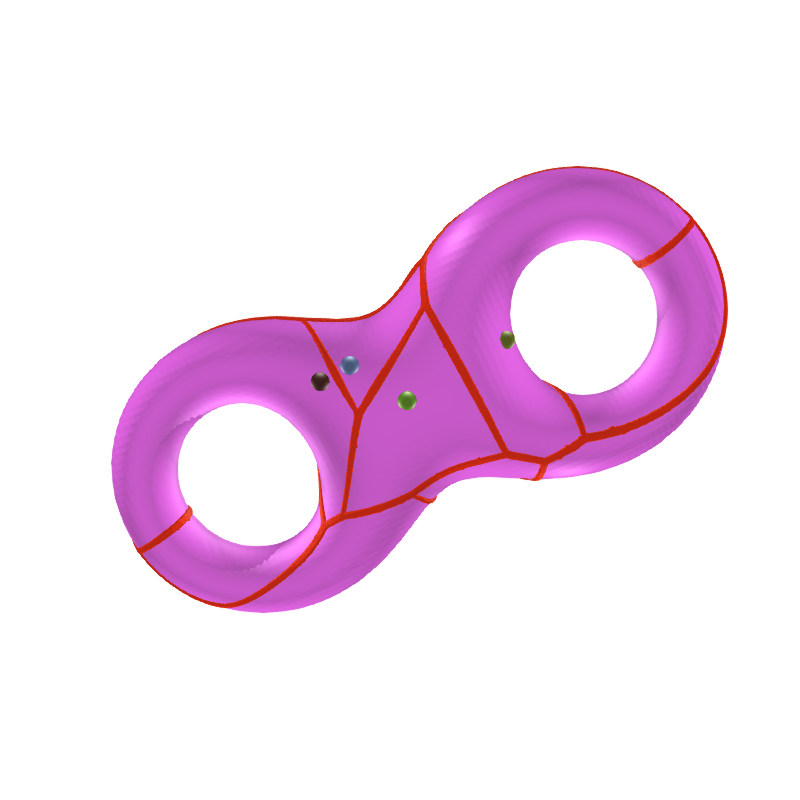} &
    \includegraphics[width=\widthvoronoi cm]{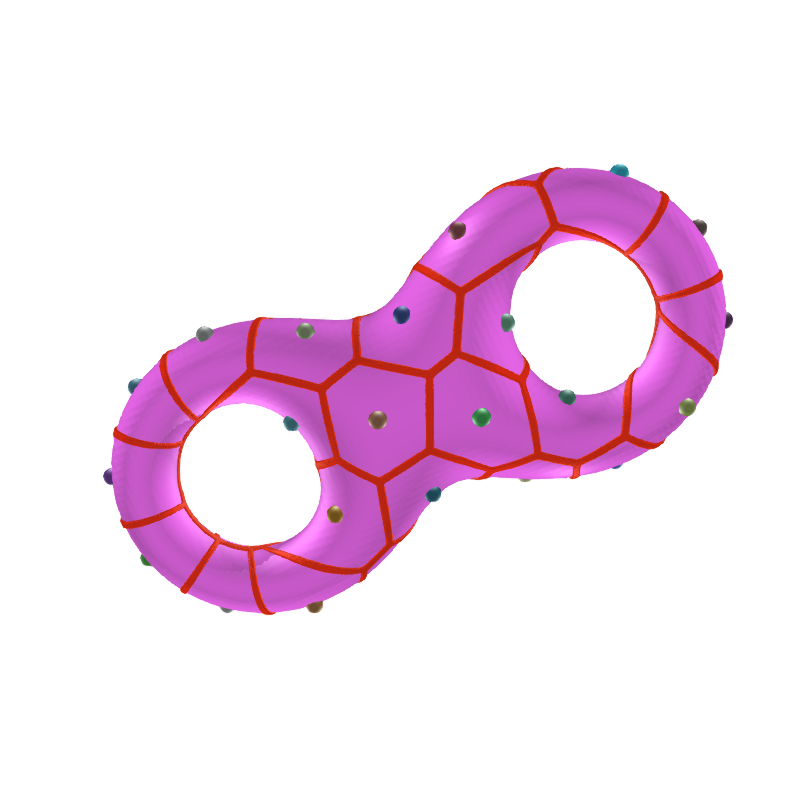} &
    \includegraphics[width=\widthvoronoi cm]{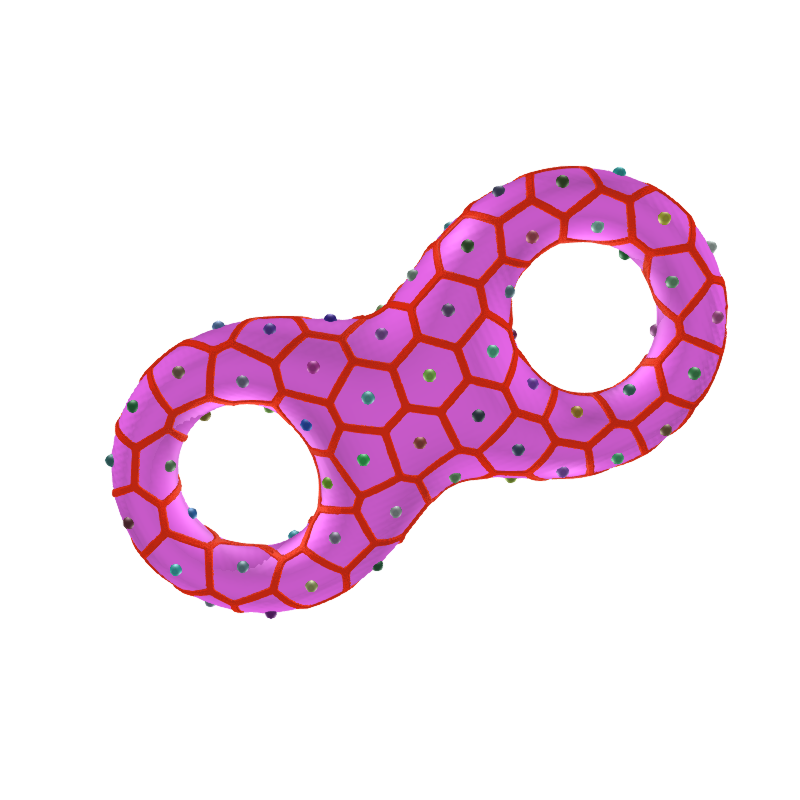}\\
    \includegraphics[width=\widthvoronoi cm]{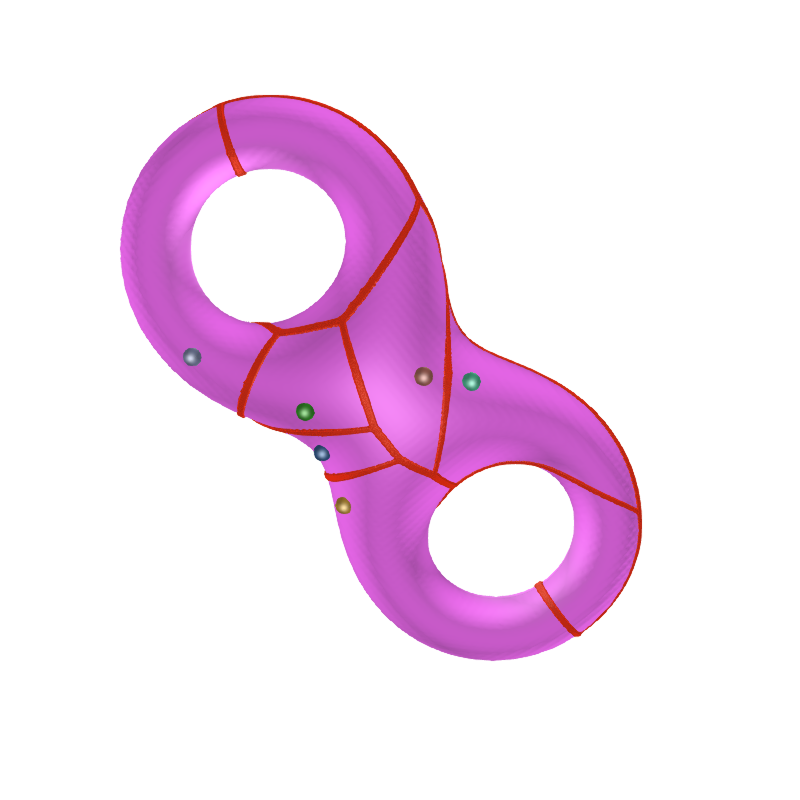} &
    \includegraphics[width=\widthvoronoi cm]{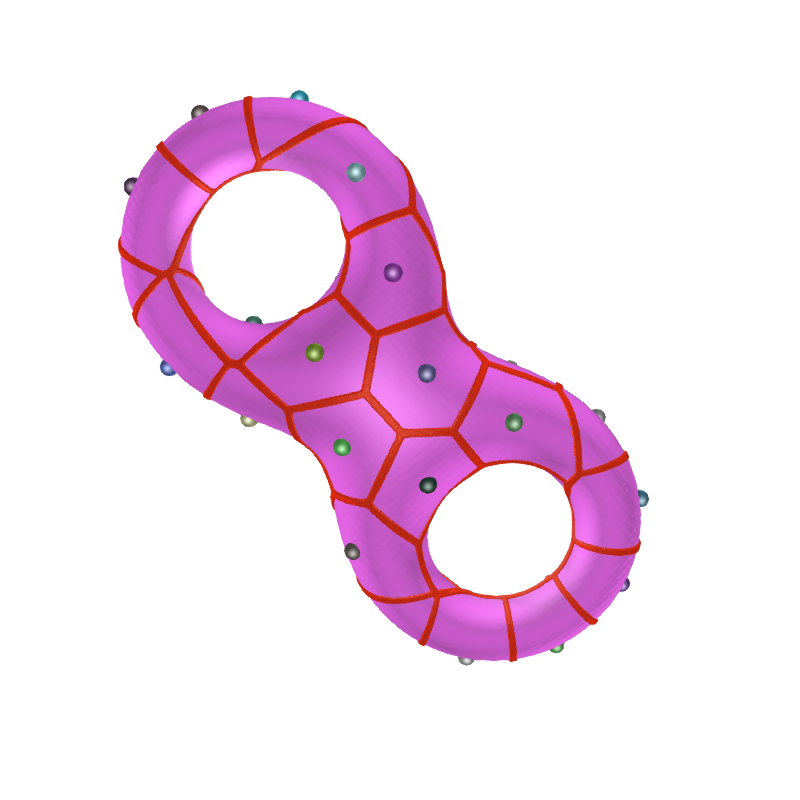} &
    \includegraphics[width=\widthvoronoi cm]{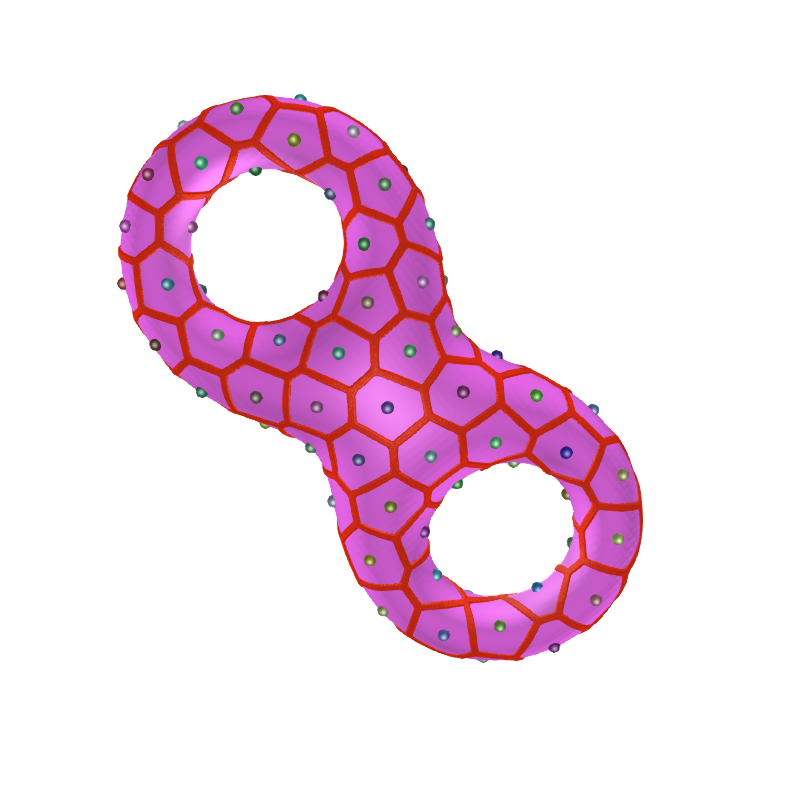}
  \end{tabular}
  \caption{Approximation of the voronoi cells on a torus of genus $2$ of $10$, $30$ and $100$ points. Every column represent two different views}
  \label{fig:vortorus1}
\end{figure}

\begin{figure}[htbp]
  \begin{tabular}{ccc}
    \includegraphics[width=\widthvoronoi cm]{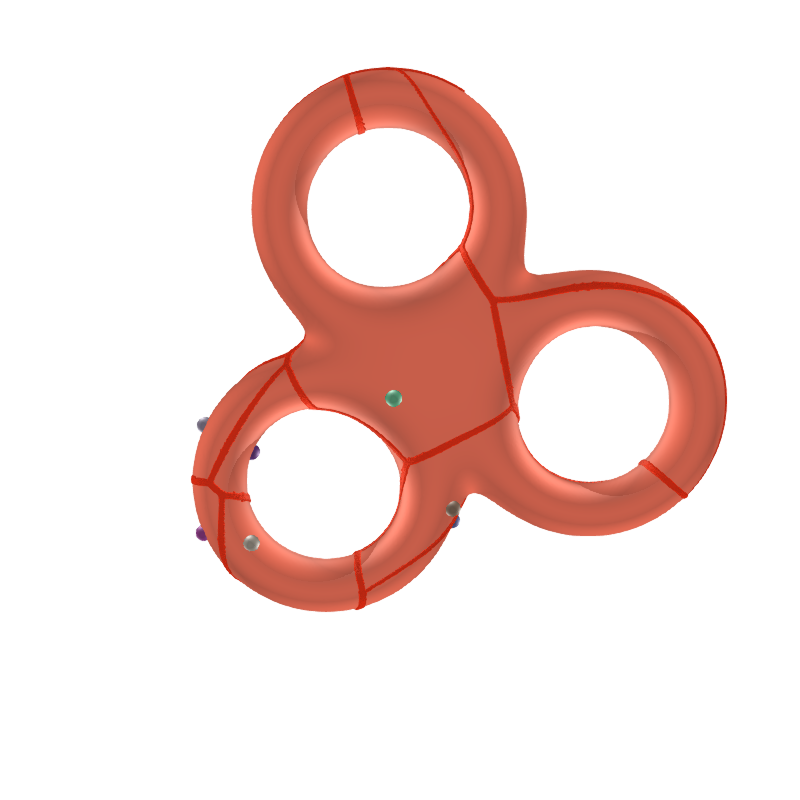} &
    \includegraphics[width=\widthvoronoi cm]{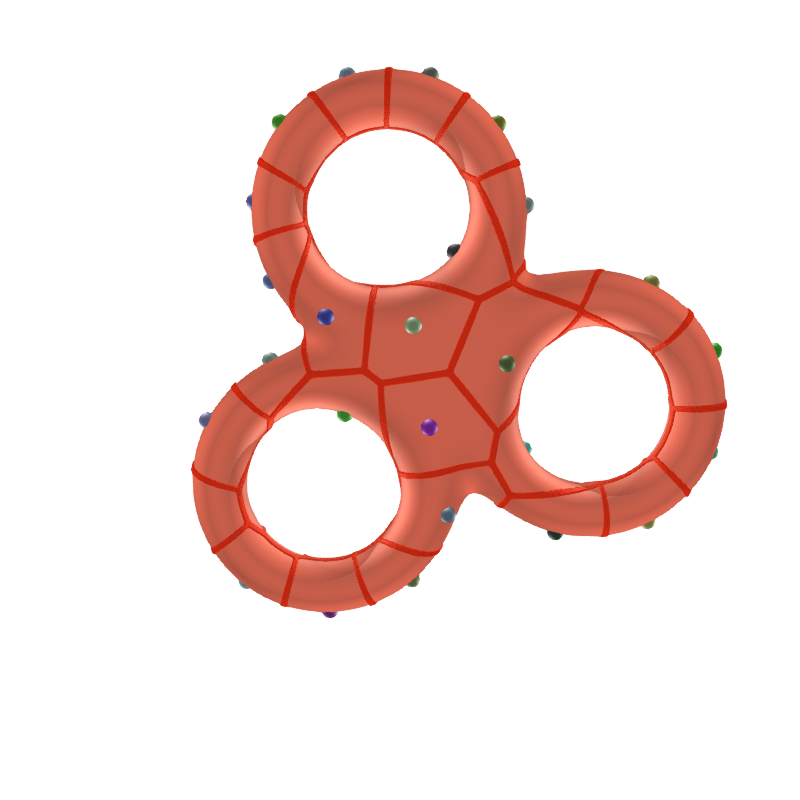} &
    \includegraphics[width=\widthvoronoi cm]{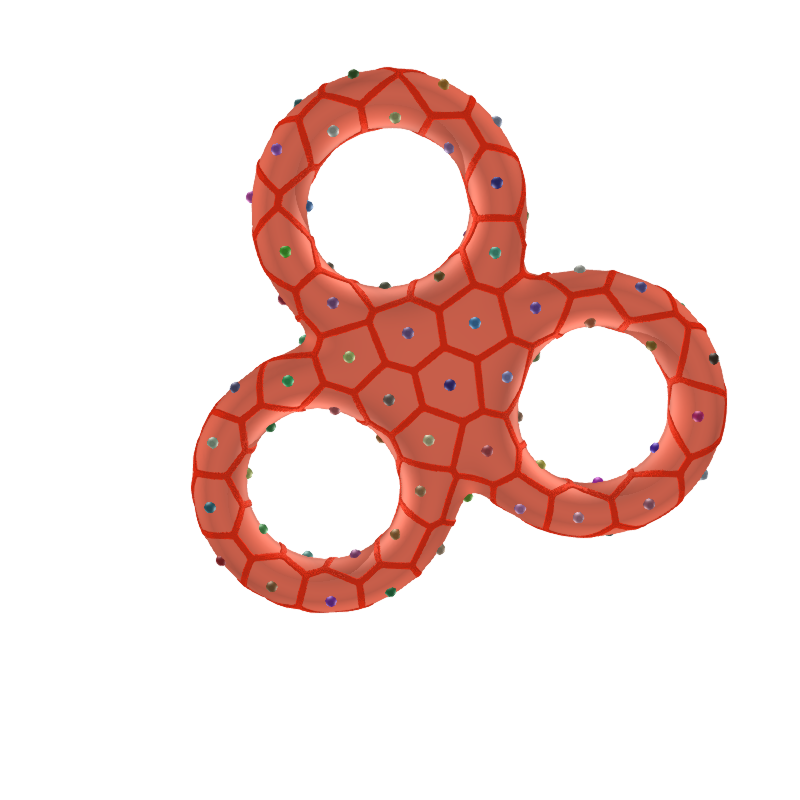}\\
    \includegraphics[width=\widthvoronoi cm]{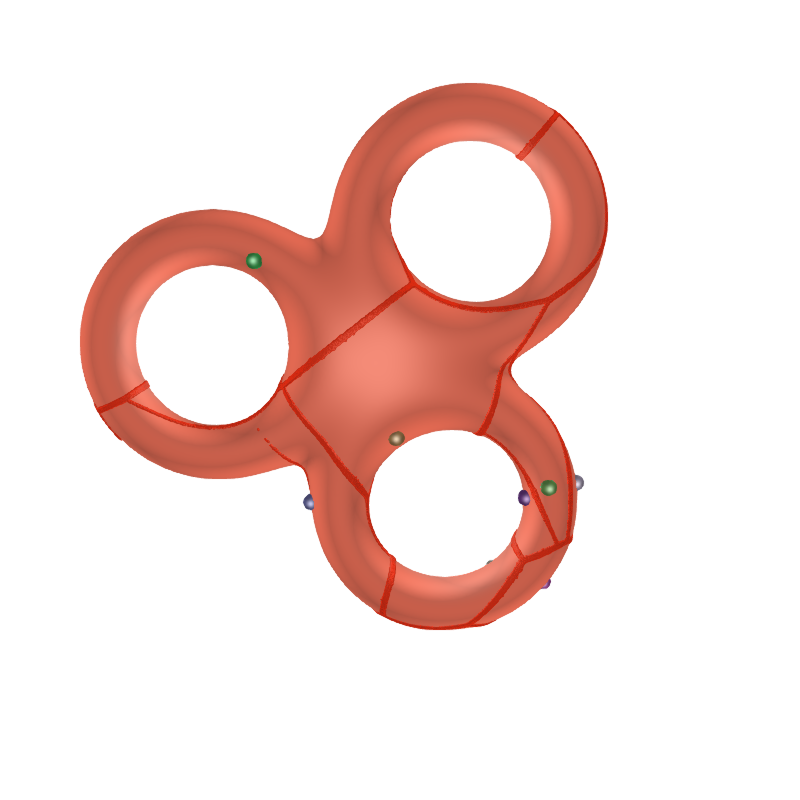} &
    \includegraphics[width=\widthvoronoi cm]{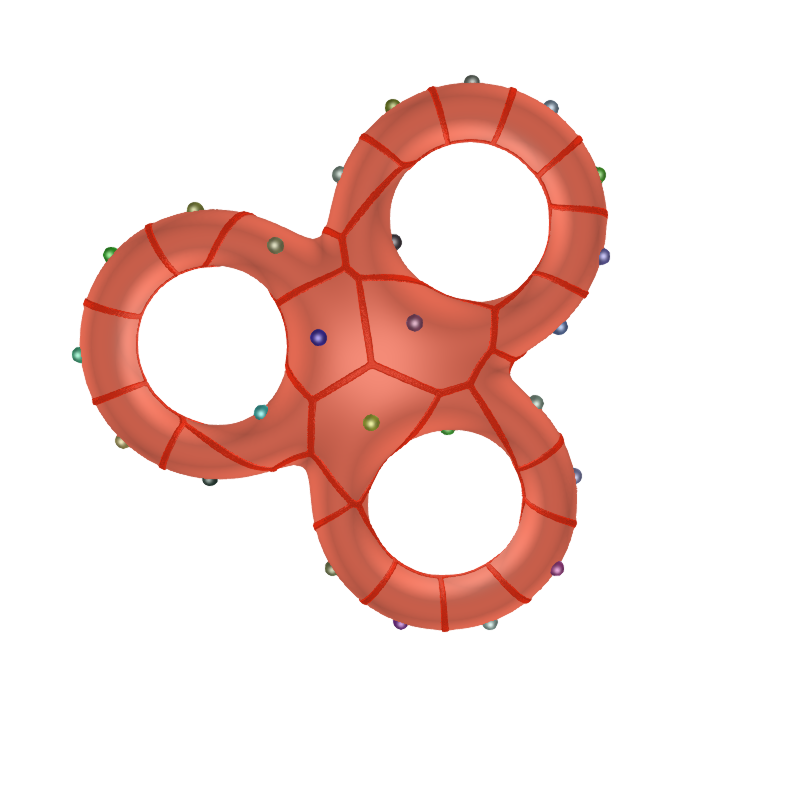} &
    \includegraphics[width=\widthvoronoi cm]{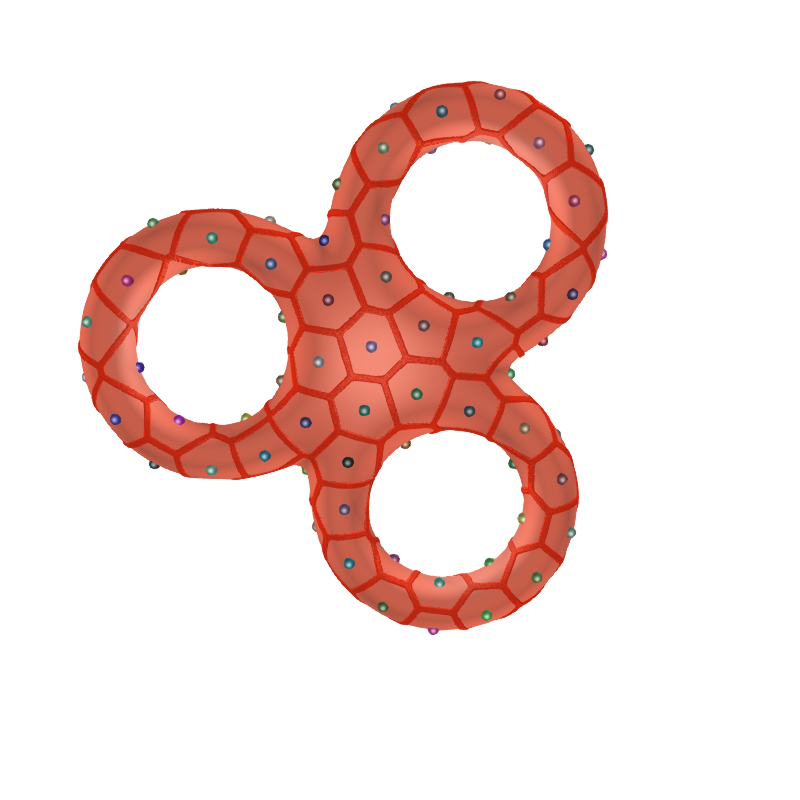}
  \end{tabular}
  \caption{Approximation of the voronoi cells on a torus of genus $3$ of $10$, $30$
  and $100$ points. Every column represent two different views}
  \label{fig:vortorus2}
\end{figure}

\FloatBarrier

%
\bibliographystyle{plain}
\bibliography{../../biblio/biblio.bib}

\bigskip\noindent
François Générau:
Laboratoire Jean Kuntzmann (LJK),
Universit\'e Joseph Fourier\\
Bâtiment IMAG, 700 avenue centrale,
38041 Grenoble Cedex 9 - FRANCE\\
{\tt francois.generau@univ-grenoble-alpes.fr}

\bigskip\noindent
\'Edouard Oudet:\\
Laboratoire Jean Kuntzmann (LJK),
Universit\'e Grenoble Alpes\\
Bâtiment IMAG, 700 avenue centrale,
38041 Grenoble Cedex 9 - FRANCE\\
{\tt edouard.oudet@univ-grenoble-alpes.fr}

\bigskip\noindent
Bozhidar Velichkov:\\
Dipartimento di Matematica, Universit\`a di Pisa\\
Largo Bruno Pontecorvo, 5, 56127 Pisa - ITALY\\
{\tt bozhidar.velichkov@gmail.com}

\end{document}